\documentclass[leqno]{siamltex704}
\usepackage{graphicx}
\usepackage{mathrsfs}
\usepackage{rotating}
\usepackage{lscape}
\usepackage{float}
\usepackage{amsfonts,amssymb}
\usepackage{dsfont}
\usepackage{pifont}
\usepackage{amsmath}
\usepackage{tikz}
\usetikzlibrary{shapes.geometric, arrows.meta, positioning}
\usetikzlibrary{calc}
\usepackage{tikz-3dplot}
\usetikzlibrary{3d,calc}
\usepackage{wrapfig} % Allows in-line images
\usepackage{hyperref}
\usepackage{multirow}
\usepackage{bm}
\usepackage{mathtools}
\usepackage{cases}
\usepackage{subfigure}
\usepackage{colortbl}
\usepackage{enumitem}
\definecolor{mygray}{gray}{.9}
\definecolor{mypink}{rgb}{.99,.91,.95}
\definecolor{mycyan}{cmyk}{.3,0,0,0}
\newcommand{\vertiii}[1]{{\left\vert\kern-0.25ex\left\vert\kern-0.25ex\left\vert #1
    \right\vert\kern-0.25ex\right\vert\kern-0.25ex\right\vert}}
\def\Cof{\operatorname{Cof}}
\def\det{\operatorname{det}}

\def\bn{{\boldsymbol n}}

\newcommand\subsetsim{\mathrel{%
  \ooalign{\raise0.2ex\hbox{$\subset$}\cr\hidewidth\raise-0.8ex\hbox{\scalebox{0.9}{$\sim$}}\hidewidth\cr}}}

\def\bn{{\boldsymbol n}}

\newcommand{\PreserveBackslash}[1]{\let\temp=\\#1\let\\=\temp}
\newcolumntype{C}[1]{>{\PreserveBackslash\centering}p{#1}}
\newcolumntype{R}[1]{>{\PreserveBackslash\raggedleft}p{#1}}
\newcolumntype{L}[1]{>{\PreserveBackslash\raggedright}p{#1}}
\def\3bar{{|\!|\!|}}
\def\bx{{\boldsymbol{x}}}
\def\bz{{\boldsymbol{z}}}
\def\bs{{\boldsymbol{s}}}

\newcommand{\dT}{\, dT}
\newcommand{\dS}{\, dS}
\title{Beyond Taylor: Divergence-Based Functional Expansions and Their Application to Numerical Integration}

\author{Junping Wang
\thanks{Directorate for Mathematical and Physical Sciences, U.S. National Science Foundation, Alexandria, VA 22314 (jwang@nsf.gov). The research of Wang was supported in part by the NSF IR/D program, while working at U.S. National Science Foundation. However, any opinion, finding, and conclusions or recommendations expressed in this material are those of the author and do not necessarily reflect the views of the U.S. National Science Foundation.}}

\begin{document}
\tikzstyle{startstop} = [rectangle, rounded corners, minimum width=3cm, 
    minimum height=1cm, text centered, draw=black, fill=gray!20]
\tikzstyle{process} = [rectangle, minimum width=3cm, minimum height=1cm, 
    text centered, draw=black, fill=blue!15]
\tikzstyle{decision} = [diamond, aspect=2, minimum width=3cm, minimum height=1cm, 
    text centered, draw=black, fill=orange!20]
\tikzstyle{arrow} = [thick, -{Stealth[length=3mm, width=2mm]}]

\maketitle

\begin{abstract}
This paper introduces a new functional expansion framework that extends classical ideas beyond the Taylor series. Unlike traditional Taylor expansions based on local polynomial approximations, the proposed approach arises from exact differential identities that link a function and its derivatives 
through polynomial weight factors. This formulation expresses smooth functions via divergence-based relations connecting derivatives of all orders with systematically scaled polynomial coefficients.
This framework provides a natural foundation for constructing high-order numerical quadrature formulas, particularly for multi-dimensional domains. By exploiting the divergence structure, volume integrals are systematically transformed into boundary integrals using the Divergence Theorem, recursively reducing the integration domain from an $n$-dimensional body to its $(n-1)$-dimensional facets, and ultimately to its vertices.

The article further enhances the framework's accuracy by introducing a complex-shift technique. It is demonstrated that by positioning the expansion center at specific roots of unity in the complex plane, lower-order error terms are annihilated, yielding high-order real-valued quadrature rules with minimal function evaluations. Additionally, a rigorous geometric analysis of the affine transformations required for surface integration is provided, deriving explicit formulas for the transformation of normal vectors and surface measures. The proposed method offers a robust, systematic, and computationally efficient alternative to tessellation-based quadrature for arbitrary flat-faced polytopes.
\end{abstract}



\pagestyle{myheadings}

\section{Introduction}

Numerical integration over multidimensional domains is a cornerstone of computational science, essential to finite element methods, weak Galerkin methods, discontinuous Galerkin schemes, and the calculation of inertial properties in computer graphics. While Gaussian quadrature provides optimal accuracy for hypercubes and simplices, integrating over general polytopes remains computationally demanding. Standard approaches typically rely on tessellation—decomposing the polytope into a union of disjoint simplices (triangles or tetrahedra). However, for complex geometries arising in Voronoi tessellations, interface reconstruction, or polyhedral finite elements, this sub-structuring can generate sliver elements that degrade numerical stability and significantly increase computational cost.

To circumvent the difficulties associated with volume meshing, methods based on dimension reduction have gained prominence. By invoking the Divergence Theorem (or Stokes' Theorem), volume integrals can be transformed into boundary integrals, recursively reducing the problem dimension until it reaches the vertices. This strategy has been explored in various forms, notably by Mirtich \cite{Mirtich1996} for mass property calculation and Lasserre \cite{Lasserre1998} for homogeneous functions. More recently, variations of this technique have been applied to integrate polynomials over arbitrary polytopes \cite{Sudhakar2014, Chin2020, Sommariva2007}.

In this work, we present a unified framework for numerical integration based on a novel functional expansion. Unlike the classical Taylor series, which expresses a function via pointwise derivatives, our expansion is derived through a recursive divergence identity. This structure naturally reformulates the integrand as a sum of terms in divergence form, allowing immediate application of the Divergence Theorem without ad-hoc algebraic manipulations.

Our contribution is threefold. First, we derive a general functional expansion in $\mathbb{R}^n$ that expresses a smooth function as a series of divergence terms. This leads to a recursive integration algorithm that operates purely on the boundary data of the polytope, completely avoiding internal discretization.

Second, we introduce a method to enhance the accuracy of these quadrature rules by extending the expansion center into the complex plane. In the one-dimensional case, we demonstrate that by selecting a complex shift based on roots of unity, we can construct high-order integration rules with a minimal number of function evaluations. For instance, we derive a real-valued quadrature rule exact for cubic polynomials using a complex-shifted trapezoidal-like formulation.

Third, we provide a rigorous geometric analysis of the affine transformations required for surface integration. While the concept of mapping a reference element to a physical facet is standard, the transformation of normal vectors and surface measures is often treated heuristically. We establish precise identities relating the surface measures and normal vectors between the reference and physical configurations, ensuring the method's implementation is mathematically robust.

The remainder of this paper is organized as follows. Section \ref{Section:2} introduces the fundamental functional expansion in one dimension, and Section \ref{Section:3} demonstrates the complex shift technique for improving quadrature accuracy. Section \ref{Section:4} extends this expansion to $n$-dimensional Euclidean space using a divergence-based recursion. Section \ref{Section:5} details the restriction of this expansion to flat surfaces and derives the affine transformation laws. Section \ref{Section:6} demonstrates the application of the divergence-based functional expansion to numerical integration in multidimensional domains. Finally, Sections \ref{Section:7} and \ref{Section:8} provide rigorous proofs for the volume of parallelepipeds and the transformation of normal vectors, verifying the geometric consistency of the proposed framework.

\section{The One-Dimensional Case: Functional Expansion}\label{Section:2}

Let $f \in C^\infty(\mathbb{R})$ be a smooth function. For any integer $k \geqslant 0$, the following identity holds:
\begin{equation}\label{eq:basic_identity_00}
    (k+1) x^k f^{(k)} = (x^{k+1}f^{(k)})' - x^{k+1} f^{(k+1)}.
\end{equation}
Rearranging terms, we obtain the equivalent form:
\begin{equation}\label{eq:basic_identity}
    x^k f^{(k)} = \frac{1}{k+1}(x^{k+1}f^{(k)})' - \frac{1}{k+1} x^{k+1} f^{(k+1)}.
\end{equation}

Repeated application of \eqref{eq:basic_identity} for $k=0, 1, 2, \dots$ leads to a functional expansion of $f$:
\begin{equation}\label{eq:expansion_derivation}
\begin{aligned}
    f &= (xf)' - x f' \\
      &= (xf)' - \frac{1}{2!} (x^2 f')' + \frac{1}{2!} x^2 f'' \\
      &= (xf)' - \frac{1}{2!} (x^2 f')' + \frac{1}{3!} (x^3 f'')' - \frac{1}{3!} x^3 f''' \\
      &= (xf)' - \frac{1}{2!} (x^2 f')' + \frac{1}{3!} (x^3 f'')' - \frac{1}{4!} (x^4 f''')' + \frac{1}{4!} x^4 f'''' \\
      &= \sum_{k=0}^m \frac{(-1)^k}{(k+1)!} \bigl( x^{k+1}f^{(k)}(x) \bigr)' + \frac{(-1)^{m+1}}{(m+1)!} x^{m+1}f^{(m+1)}(x).
\end{aligned}
\end{equation}
The result can be summarized as follows.

\begin{theorem}\label{thm:finite_expansion}
Let $m \geqslant 0$ be an integer and $f \in C^{m+1}(\mathbb{R})$. Then
\begin{equation}\label{eq:finite_sum}
    f(x) = \sum_{k=0}^m \frac{(-1)^k}{(k+1)!} \Bigl( x^{k+1}f^{(k)}(x) \Bigr)' + \frac{(-1)^{m+1}}{(m+1)!} x^{m+1}f^{(m+1)}(x).
\end{equation}
Moreover, if $f \in C^\infty(\mathbb{R})$, we have the infinite expansion
\begin{equation}\label{eq:infinite_sum}
    f(x) = \sum_{k=0}^\infty \frac{(-1)^k}{(k+1)!} \Bigl( x^{k+1}f^{(k)}(x) \Bigr)',
\end{equation}
provided the series converges.
\end{theorem}

\medskip

Next, for any fixed point $x_0$ (real or complex), replacing $x$ by $x-x_0$ in \eqref{eq:basic_identity_00} gives
\[
    (k+1)(x-x_0)^k f^{(k)} = \bigl((x-x_0)^{k+1}f^{(k)}\bigr)' - (x-x_0)^{k+1} f^{(k+1)}.
\]
Consequently, we have
\begin{equation}\label{eq:shifted_identity}
    (x-x_0)^k f^{(k)} = \frac{1}{k+1}\bigl((x-x_0)^{k+1}f^{(k)}\bigr)' - \frac{1}{k+1} (x-x_0)^{k+1} f^{(k+1)},
\end{equation}
which yields the following generalized result.

\begin{theorem}\label{thm:shifted_expansion}
Let $m \geqslant 0$ be an integer and $x_0$ be a given real or complex number. If $f \in C^{m+1}(\mathbb{R})$, then
\begin{equation}\label{eq:shifted_finite}
    f(x) = \sum_{k=0}^m \frac{(-1)^k}{(k+1)!} \Bigl( (x-x_0)^{k+1}f^{(k)}(x) \Bigr)' + \frac{(-1)^{m+1}}{(m+1)!} (x-x_0)^{m+1}f^{(m+1)}(x).
\end{equation}
Furthermore, for $f \in C^\infty(\mathbb{R})$, the following series expansion holds:
\begin{equation}\label{eq:shifted_infinite}
    f(x) = \sum_{k=0}^\infty \frac{(-1)^k}{(k+1)!} \Bigl( (x-x_0)^{k+1}f^{(k)}(x) \Bigr)',
\end{equation}
provided it converges.
\end{theorem}

\section{Numerical Integration in One Dimension}\label{Section:3}

In this section, we employ the expansions \eqref{eq:shifted_finite}-\eqref{eq:shifted_infinite} to derive several numerical integration formulas in one dimension. These formulas are presented primarily for illustration, as many other numerical integration schemes can be obtained from this expansion through varying choices of $x_0$ and truncation.

For a given interval $[a,b]$, integrating the expansion \eqref{eq:shifted_infinite} over $[a, b]$ yields
\begin{equation}\label{eq:int_general}
    \int_a^b f(x) \, dx = \sum_{k=0}^\infty \frac{(-1)^k}{(k+1)!} \Bigl( (b-x_0)^{k+1} f^{(k)}(b) - (a-x_0)^{k+1}f^{(k)}(a) \Bigr).
\end{equation}
Using the first $m+1$ terms of this series provides an approximation for the integral:
\begin{equation}\label{eq:int_approx}
    \int_a^b f(x) \, dx \approx I_{m+1} := \sum_{k=0}^{m} \frac{(-1)^k}{(k+1)!} \Bigl( (b-x_0)^{k+1} f^{(k)}(b) - (a-x_0)^{k+1}f^{(k)}(a) \Bigr).
\end{equation}
From the identity \eqref{eq:shifted_finite}, the corresponding remainder is
\begin{equation}\label{eq:remainder_general}
    R_{m+1} = \frac{(-1)^{m+1}}{(m+1)!} \int_a^b (x-x_0)^{m+1}f^{(m+1)}(x) \, dx.
\end{equation}

By setting $x_0=b$, formula \eqref{eq:int_general} simplifies to
\begin{equation}\label{eq:taylor_a}
    \int_a^b f(x) \, dx = \sum_{k=0}^\infty \frac{(b-a)^{k+1}}{(k+1)!} f^{(k)}(a), 
\end{equation}
which coincides with the result obtained from the standard Taylor expansion about $x=a$. The corresponding approximation is
\begin{equation}\label{eq:taylor_a_approx}
    \int_a^b f(x) \, dx \approx \sum_{k=0}^m \frac{(b-a)^{k+1}}{(k+1)!} f^{(k)}(a), 
\end{equation}
with remainder
\begin{equation}\label{eq:taylor_a_rem}
    R_{m+1} = \frac{(-1)^{m+1}}{(m+1)!} \int_a^b (x-b)^{m+1}f^{(m+1)}(x) \, dx.
\end{equation}

Similarly, choosing $x_0=a$ gives the expansion about the right endpoint:
\begin{equation}\label{eq:taylor_b}
    \int_a^b f(x) \, dx = \sum_{k=0}^\infty \frac{(-1)^k(b-a)^{k+1}}{(k+1)!} f^{(k)}(b), 
\end{equation}
with the approximation
\begin{equation}\label{eq:taylor_b_approx}
    \int_a^b f(x) \, dx \approx \sum_{k=0}^m \frac{(-1)^k(b-a)^{k+1}}{(k+1)!} f^{(k)}(b), 
\end{equation}
and remainder 
\begin{equation}\label{eq:taylor_b_rem}
    R_{m+1} = \frac{(-1)^{m+1}}{(m+1)!} \int_a^b (x-a)^{m+1}f^{(m+1)}(x) \, dx.
\end{equation}

To derive a numerical integration formula involving only the function value at the midpoint $c = \frac{1}{2}(a+b)$, one may apply \eqref{eq:taylor_b_approx} to the sub-integral $\int_{a}^c f(x) \, dx$ and \eqref{eq:taylor_a_approx} to $\int_{c}^b f(x) \, dx$, subsequently combining the two results. The details of this exploration are left to the reader.

Alternatively, selecting the midpoint $x_0 = \frac{1}{2}(a+b)$ leads to 
\begin{equation}\label{eq:midpoint_expansion}
    \int_a^b f(x) \, dx = \sum_{k=0}^\infty \frac{(-1)^k}{(k+1)!} \frac{(b-a)^{k+1}}{2^{k+1}} \Bigl[ f^{(k)}(b) + (-1)^k f^{(k)}(a) \Bigr],
\end{equation}
with the approximation
\begin{equation}\label{eq:midpoint_approx}
    \int_a^b f(x) \, dx \approx \sum_{k=0}^m \frac{(-1)^k}{(k+1)!} \frac{(b-a)^{k+1}}{2^{k+1}} \Bigl[ f^{(k)}(b) + (-1)^k f^{(k)}(a) \Bigr],
\end{equation}
and remainder
\begin{equation}\label{eq:midpoint_rem}
    R_{m+1} = \frac{(-1)^{m+1}}{(m+1)!} \int_a^b \left(x-\frac{a+b}{2}\right)^{m+1}f^{(m+1)}(x) \, dx.
\end{equation}

Retaining only the first term of \eqref{eq:midpoint_approx} yields the {\em Trapezoidal rule}: 
\[
    \int_a^b f(x) \, dx \approx \frac{b-a}{2} \bigl( f(a) + f(b) \bigr),
\]
whose remainder, obtained from \eqref{eq:midpoint_rem} with $m=0$, is
\begin{equation*}
    R_{1} = - \int_a^b \left(x-\frac{a+b}{2}\right)f'(x) \, dx,
\end{equation*}
which vanishes for any linear function $f$. 

Including the first three terms of \eqref{eq:midpoint_approx} provides a higher-order approximation:
\begin{equation}\label{eq:higher_order_3}
\begin{aligned}
    \int_a^b f(x) \, dx \approx \,\, &\frac{b-a}{2} \bigl( f(a)+f(b) \bigr) - \frac{(b-a)^2}{8} \bigl( f'(b)-f'(a) \bigr) \\
    &+ \frac{(b-a)^3}{48} \bigl( f''(b)+f''(a) \bigr),
\end{aligned}
\end{equation}
whose remainder, from \eqref{eq:midpoint_rem} with $m=2$, is
\begin{equation}\label{eq:higher_order_3_rem}
    R_{3} = -\frac{1}{3!} \int_a^b \left(x-\frac{a+b}{2}\right)^{3}f'''(x) \, dx.
\end{equation}
This remainder vanishes for all cubic polynomials $f$; thus, formula \eqref{eq:higher_order_3} is exact for cubic functions.

If only the first two terms are retained, the approximation becomes
\begin{equation}\label{eq:higher_order_2}
    \int_a^b f(x) \, dx \approx \frac{b-a}{2} \bigl( f(a)+f(b) \bigr) - \frac{(b-a)^2}{8} \bigl( f'(b)-f'(a) \bigr),
\end{equation}
with remainder
\begin{equation}\label{eq:higher_order_2_rem}
    R_{2} = \frac{1}{2!} \int_a^b \left(x-\frac{a+b}{2}\right)^{2}f''(x) \, dx.
\end{equation}
The remainder $R_2$ vanishes for linear functions but generally not for quadratic ones; hence, the two-term approximation \eqref{eq:higher_order_2} remains exact only for linear functions.

\subsection{Improving Accuracy via a Complex Shift}\label{subsection_complexshift}

Can the accuracy of \eqref{eq:higher_order_2} be improved while still using only the first two terms of the expansion—ideally making it exact for quadratic polynomials? The answer is affirmative in one dimension by introducing an asymmetric complex shift of the expansion point $x_0$. The core idea is to select a complex number $x_0$ such that the contribution of the remainder term \eqref{eq:remainder_general} with $m=2$ vanishes for all quadratic polynomials. This condition is satisfied when
\[
(b-x_0)^3 = (a-x_0)^3 \implies b-x_0 = \omega_3 (a-x_0),
\] 
where $\omega_3$ denotes a nontrivial cubic root of unity (i.e., $\omega_3^3=1$ and $\omega_3 \neq 1$). 

Solving for $x_0$ yields:
\[
x_0 = \frac{b-\omega_3 a}{1-\omega_3}, \quad b-x_0 = \frac{\omega_3 (b-a)}{\omega_3-1}, \quad a-x_0 = \frac{b-a}{\omega_3-1}.
\]
Substituting this $x_0$ into the two-term formula gives a complex-valued approximation:
\begin{equation}
    \int_a^b f(x) \, dx \approx \frac{b-a}{\omega_3-1} \bigl( \omega_3 f(b) - f(a) \bigr) - \frac{(b-a)^2}{2(\omega_3-1)^2} \bigl( \omega_3^2 f'(b) - f'(a) \bigr).
\end{equation}
Taking the real part produces a real-valued high-order integration rule that is exact for all quadratic polynomials.  

To compute the real part of $\frac{\omega_3}{\omega_3-1}$, observe that:
\begin{equation}\label{eq:real_part_calc}
    \frac{\omega_3}{\omega_3-1} = \frac{\omega_3(\bar\omega_3 -1)}{(\omega_3-1)(\bar\omega_3 -1)} = \frac{1-\omega_3}{2(1 - \Re(\omega_3))}.
\end{equation}
Hence, $\Re\bigl(\frac{\omega_3}{\omega_3-1}\bigr) = \frac{1}{2}$. It follows that:
\begin{equation}
    \Re\left( \frac{1}{\omega_3-1} \right) = \Re\left( \frac{\omega_3}{\omega_3-1} \right) - 1 = -\frac{1}{2}.
\end{equation}

For the derivative term, we use the identity
\[
 \frac{\omega_3^2}{(\omega_3-1)^2} - \frac{1}{(\omega_3-1)^2} = \frac{\omega_3+1}{\omega_3-1} = \frac{\omega_3}{\omega_3-1} + \frac{1}{\omega_3-1},
\]
which implies $\Re\bigl( \frac{\omega_3^2}{(\omega_3-1)^2} \bigr) = \Re\bigl(\frac{1}{(\omega_3-1)^2}\bigr)$. By letting $A = 1 - \Re(\omega_3)$, we have from \eqref{eq:real_part_calc} that $\frac{\omega_3}{\omega_3-1} = \frac{1-\omega_3}{2A}$. Note that $\Re((1-\omega_3)^2) = -2A \Re(\omega_3)$. Thus:
\[
    \Re\left( \frac{\omega_3^2}{(\omega_3-1)^2} \right) = \frac{1}{4A^2} \Re\bigl((1-\omega_3)^2\bigr) = \frac{-\Re(\omega_3)}{2(1-\Re(\omega_3))}.
\]
For $\omega_3 = -\frac{1}{2} + \frac{\sqrt{3}}{2}i$, we find $\Re\bigl( \frac{\omega_3^2}{(\omega_3-1)^2} \bigr) = \frac{1}{6}$. The resulting real-valued quadrature rule is:
\begin{equation}\label{eq:quad_exact}
    \int_a^b f(x) \, dx \approx \frac{b-a}{2} \bigl( f(b) + f(a) \bigr) - \frac{(b-a)^2}{12} \bigl( f'(b) - f'(a) \bigr),
\end{equation}
which is exact for all quadratic polynomials. This generalizes the trapezoidal rule by incorporating derivative information and optimizing expansion through complex symmetry.

\subsection{Four-term Formula via a Complex Shift}\label{subsection_complexshift4term}

The complex shift approach can be extended to formulas involving any even number of terms. For the four-term formula, we select $x_0$ such that the remainder term vanishes for all polynomials of degree 4:
\begin{equation}\label{eq:remainder_vanish_4}
    R_{4} = \frac{1}{4!} \int_a^b (x-x_0)^{4} f^{(4)}(x) \, dx = 0.
\end{equation}
This condition is satisfied when $(b-x_0)^5 = (a-x_0)^5$, implying $b-x_0 = \omega_5 (a-x_0)$ where $\omega_5 \neq 1$ is a 5th root of unity. Solving for $x_0$ gives $x_0 = \frac{b-\omega_5 a}{1-\omega_5}$. Defining $\gamma_5 = \frac{\omega_5}{\omega_5-1}$ (the Möbius transform of $\omega_5$), we have:
\[
    b-x_0 = \gamma_5 (b-a), \quad a-x_0 = (\gamma_5-1)(b-a).
\]
Since the Möbius transform maps the unit circle to the vertical line $\Re(z) = \frac{1}{2}$, it follows that $1-\gamma_5 = \bar{\gamma}_5$, leading to $\Re((1-\gamma_5)^k) = \Re(\gamma_5^k)$ for any integer $k$.

The resulting real-valued integration rule, exact for polynomials of degree 4, is:
\begin{equation}
\begin{aligned}
    \int_a^b f(x) \, dx \approx \,\, & \frac{b-a}{2} \bigl(f(b) + f(a)\bigr) - \frac{(b-a)^2 \Re(\gamma_5^2)}{2!} \bigl( f'(b) - f'(a) \bigr) \\
    & + \frac{(b-a)^3 \Re(\gamma_5^3)}{3!} \bigl( f''(b) + f''(a) \bigr) \\
    & - \frac{(b-a)^4 \Re(\gamma_5^4)}{4!} \bigl( f'''(b) - f'''(a) \bigr).
\end{aligned}
\end{equation}

Table \ref{Table1} provides the real parts of $\gamma_5^m$ for different choices of $\omega_5$.

\begin{table}[ht]
\centering
\caption{Real parts of $\gamma_5^m$ for nontrivial 5th roots of unity}\label{Table1}
\renewcommand{\arraystretch}{1.5}
\begin{tabular}{ccc}
\hline
\textbf{Term} & $\omega_5 = e^{\pm \frac{4\pi i}{5}}$ & $\omega_5 = e^{\pm \frac{2\pi i}{5}}$ \\ \hline
$\Re(\gamma_5)$ & $1/2$ & $1/2$ \\
$\Re(\gamma_5^2)$ & $\frac{\sqrt{5}}{10}$ & $-\frac{\sqrt{5}}{10}$ \\
$\Re(\gamma_5^3)$ & $\frac{3\sqrt{5}-5}{20}$ & $-\frac{5+3\sqrt{5}}{20}$ \\
$\Re(\gamma_5^4)$ & $\frac{\sqrt{5}-2}{10}$ & $-\frac{\sqrt{5}+2}{10}$ \\ \hline
\end{tabular}
\end{table}

\section{Functional Expansion in $n$ Dimensions}\label{Section:4}

In this section, we consider functions defined in the Euclidean space $\mathbb{R}^n$, with $\mathbf{x}=(x_1, \dots, x_n)$ denoting a point in $\mathbb{R}^n$. Let $f = f(\mathbf{x})$ be a smooth function in $\mathbb{R}^n$. For any integer $k \ge 0$, we introduce the following notation for the $k$-th order directional derivative:
\begin{equation}
    \nabla^k f : \mathbf{x}^k := \sum_{1 \le j_r \le n} x_{j_1} x_{j_2} \dots x_{j_k} \partial^k_{j_1, j_2, \dots, j_k} f.
\end{equation}
For simplicity, the multi-index summation symbol will be omitted in the subsequent derivations.

\begin{theorem}\label{thm:grad_identity}
For any $k \ge 0$ and $\mathbf{z}_0 \in \mathbb{C}^n$, the following identity holds:
\begin{equation}\label{eq:grad_step}
    (\mathbf{x}-\mathbf{z}_0) \cdot \nabla \bigl( \nabla^k f : (\mathbf{x}-\mathbf{z}_0)^k \bigr) = k \nabla^k f : (\mathbf{x}-\mathbf{z}_0)^k + \nabla^{k+1} f : (\mathbf{x}-\mathbf{z}_0)^{k+1}.
\end{equation}
\end{theorem}

\begin{proof}
Without loss of generality, we prove the result for $\mathbf{z}_0 = 0$. By the definition of the gradient operator and the product rule, we have: 
\begin{equation*}
\begin{aligned}
   &  \mathbf{x} \cdot \nabla (\nabla^k f : \mathbf{x}^k) \\
= &\sum_{j_{k+1}=1}^n x_{j_{k+1}} \partial_{j_{k+1}}(\nabla^k f : \mathbf{x}^k) \\
    = & x_{j_1} \dots x_{j_{k+1}} \partial^{k+1}_{j_1, \dots, j_{k+1}} f + \sum_{j_{k+1}=1}^n x_{j_{k+1}}  \sum_{m=1}^k x_{j_1} \dots \delta_{j_m, j_{k+1}} \dots x_{j_k} \partial^k_{j_1, \dots, j_k} f\\
    = &  x_{j_1} \dots x_{j_{k+1}} \partial^{k+1}_{j_1, \dots, j_{k+1}} f + k  x_{j_1} \dots x_{j_k} \partial^k_{j_1, \dots, j_k} f \\
    = & k \nabla^k f : \mathbf{x}^k + \nabla^{k+1} f : \mathbf{x}^{k+1}.
\end{aligned}
\end{equation*}
\end{proof}

Next, we consider the divergence of the vector-valued function $\mathbf{x} (\nabla^k f : \mathbf{x}^k)$. Using the identity $\nabla \cdot (\mathbf{u} \phi) = \phi \nabla \cdot \mathbf{u} + \mathbf{u} \cdot \nabla \phi$ and equation \eqref{eq:grad_step}, we obtain:
\begin{equation*}
\begin{aligned}
    \nabla \cdot (\mathbf{x} \nabla^k f : \mathbf{x}^k) &= (\nabla \cdot \mathbf{x}) \nabla^k f : \mathbf{x}^k + \mathbf{x} \cdot \nabla (\nabla^k f : \mathbf{x}^k) \\
    &= n \nabla^k f : \mathbf{x}^k + k \nabla^k f : \mathbf{x}^k + \nabla^{k+1} f : \mathbf{x}^{k+1} \\
    &= (n+k) \nabla^k f : \mathbf{x}^k + \nabla^{k+1} f : \mathbf{x}^{k+1}.
\end{aligned}
\end{equation*}

\begin{theorem}
For any $k \ge 0$ and $\mathbf{z}_0 \in \mathbb{C}^n$, the recursive expansion identity is:
\begin{equation}\label{eq:div_recursion}
    \nabla^k f : (\mathbf{x}-\mathbf{z}_0)^k = \frac{1}{n+k} \nabla \cdot \bigl( (\mathbf{x}-\mathbf{z}_0) \nabla^k f : (\mathbf{x}-\mathbf{z}_0)^k \bigr) - \frac{1}{n+k} \nabla^{k+1} f : (\mathbf{x}-\mathbf{z}_0)^{k+1}.
\end{equation}
\end{theorem}

Repeated application of \eqref{eq:div_recursion} yields the following functional expansion (illustrated for $\mathbf{z}_0 = 0$):
\begin{equation*}
\begin{aligned}
    f(\mathbf{x}) &= \frac{1}{n} \nabla \cdot (\mathbf{x} f) - \frac{1}{n} \nabla f : \mathbf{x} \\
    &= \frac{1}{n} \nabla \cdot (\mathbf{x} f) - \frac{1}{n(n+1)} \nabla \cdot (\mathbf{x} \nabla f : \mathbf{x}) + \frac{1}{n(n+1)} \nabla^2 f : \mathbf{x}^2 \\
    &= (n-1)! \sum_{k=0}^\infty \frac{(-1)^k}{(n+k)!} \nabla \cdot \bigl( \mathbf{x} \nabla^k f : \mathbf{x}^k \bigr).
\end{aligned}
\end{equation*}

The general result is summarized as follows:
\begin{theorem} \label{thm:JW_expansion}
Let $f \in C^{m+1}(\mathbb{R}^n)$ and $\mathbf{z}_0 \in \mathbb{C}^n$. Then:
\begin{equation}\label{eq:JW_finite}
\begin{split}
    \frac{f(\mathbf{x})}{(n-1)!} = & \sum_{k=0}^m \frac{(-1)^k}{(n+k)!} \nabla \cdot \bigl( (\mathbf{x}-\mathbf{z}_0) \nabla^k f : (\mathbf{x}-\mathbf{z}_0)^k \bigr)\\
& + \frac{(-1)^{m+1}}{(n+m)!} \nabla^{m+1} f : (\mathbf{x}-\mathbf{z}_0)^{m+1}.
\end{split}
\end{equation}
If $f \in C^\infty(\mathbb{R}^n)$, the infinite series expansion
\begin{equation}\label{eq:JW_series}
    \frac{f(\mathbf{x})}{(n-1)!} = \sum_{k=0}^\infty \frac{(-1)^k}{(n+k)!} \nabla \cdot \bigl( (\mathbf{x}-\mathbf{z}_0) \nabla^k f : (\mathbf{x}-\mathbf{z}_0)^k \bigr)
\end{equation}
holds, provided the series converges.
\end{theorem}

\bigskip
The expansion in Theorem \ref{thm:JW_expansion} may be viewed as a divergence-form alternative to the classical Taylor expansion. While the standard Taylor formula represents $f(\mathbf{x})$ via pointwise derivatives evaluated at a fixed point $\mathbf{z}_0$, this expansion redistributes derivatives into divergence terms. This structure is particularly advantageous for applications involving integration by parts—such as weak formulations and finite element analysis—as the leading terms can be converted directly into boundary integrals via the divergence theorem.

\section{Functional Expansion on Flat Surfaces}\label{Section:5}
Let $\Sigma$ be an $m$-dimensional flat surface in $\mathbb{R}^n$, with $m \leqslant n$. The goal of this section is to extend the series expansion \eqref{eq:JW_finite}-\eqref{eq:JW_series} to functions defined on $\Sigma$ by utilizing the surface divergence operator.

For illustrative purposes, Figure \ref{FlatFaceSigma} depicts a two-dimensional surface $\Sigma$ in three-dimensional space. Although $\bx_0 \in \Sigma$ can be any point on the flat surface, it is often advantageous to select $\bx_0$ as the barycenter (or center) of $\Sigma$.

\begin{figure}[h!]
\centering
\begin{tikzpicture}[scale=0.7]
\coordinate (O) at (0,0); 
\coordinate (X1) at (5,0); 
\coordinate (X3) at (0,3); 
\coordinate (X2) at (-2,-2);
\coordinate (A0) at (1.7,-1.5);
\coordinate (A1) at (-0.8,-0.2);
\coordinate (A01) at (0.45,-0.85);
\coordinate (A02) at (2.85,-0.6);
\coordinate (A14) at (-0.05, 1.0);
\coordinate (A02V) at (3.315,-0.24);
\coordinate (A2) at (4, 0.3);
\coordinate (A3) at (4, 2.2);
\coordinate (A4) at (0.7, 2.2);
\coordinate (NX) at (3.45, -1.366);
\coordinate (NXL) at (-0.89, 1.45);
\coordinate (XCenter) at (2, 0.7);
\coordinate (X) at (1.5, -0.85);
\coordinate (Sigma) at (1, 1);

\draw[thick] (A0)--(A2)--(A3)--(A4)--(A1)--cycle; 
\draw[->, thick] (O)--(X1); 
\draw[->, thick] (O)--(X2); 
\draw[->, thick] (O)--(X3); 
\draw[->, thick, blue] (A14)--(NXL); 

\draw node[below] at (X) {$\bx$}; 
\draw node[right] at (X1) {$x_1$}; 
\draw node[above] at (X2) {$x_2$};
\draw node[left] at (X3) {$x_3$};
\draw node[below] at (A0) {$A_0$}; 
\draw node[left] at (NXL) {$\bn_x$};
\draw node[above, red] at (A01) {$\vec{v}_1$};
\draw node[above, red] at (A02V) {$\vec{v}_2$};
\draw node[right] at (A2) {$A_2$}; 
\draw node[left] at (A1) {$A_1$};
\draw node[above] at (A4) {$A_4$};
\draw node[above] at (A3) {$A_3$};
\draw node[above] at (XCenter) {$\bx_0$};
\draw node[below] at (Sigma) {$\Sigma$};

\draw[->, red, thick] (A0)--(A1); 
\draw[->, red, thick] (A0)--(A2);
\draw[->, blue, thick] (XCenter)--(X);

\filldraw[black] (A0) circle(0.05);
\filldraw[black] (A1) circle(0.05);
\filldraw[black] (A2) circle(0.05);
\filldraw[black] (XCenter) circle(0.05);
\filldraw[black] (A3) circle(0.05);
\filldraw[black] (A4) circle(0.05);
\end{tikzpicture}
\caption{Depiction of a flat face $\Sigma \subset \partial T$.}
\label{FlatFaceSigma}
\end{figure}
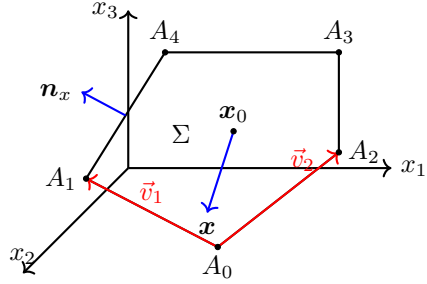

For any point $\bx \in \Sigma$, the vector from $\bx_0$ to $\bx$ can be expressed as a linear combination of the basis vectors of the hyperplane containing $\Sigma$. Specifically, let
\[
\vec{v}_1 = A_1 - A_0, \quad \vec{v}_2 = A_2 - A_0.
\]
Then there exists a column vector $\bs = (s_1, s_2)^\top \in \mathbb{R}^2$ such that
\begin{equation}\label{EQ_affine_map_2D}
\bx - \bx_0 = s_1 \vec{v}_1 + s_2 \vec{v}_2 = A \bs, \quad \bs \in \Sigma_s \subset \mathbb{R}^2,
\end{equation}
where $A = [\vec{v}_1, \vec{v}_2]$. The mapping in \eqref{EQ_affine_map_2D} defines an affine transformation from the two-dimensional reference domain $\Sigma_s$ to the flat face $\Sigma$:
\begin{equation}\label{AffineMap1}
\bs \to F(\bs) := \bx = \bx_0 + A \bs, \quad \bs \in \Sigma_s.
\end{equation}
Note that in the illustrative example of Figure \ref{FlatFaceSigma}, $A$ is a $3 \times 2$ matrix.

To generalize to an $m$-dimensional surface $\Sigma \subset \mathbb{R}^n$, the affine map \eqref{AffineMap1} naturally extends to involve an $n \times m$ matrix $A = [\vec{v}_1, \vec{v}_2, \dots, \vec{v}_m] = \{a_{ij}\}_{n \times m}$ and a reference domain $\Sigma_s \subset \mathbb{R}^m$.

Let $\phi = \phi(\bx)$ be a smooth function defined on $\mathbb{R}^n$. Its restriction to $\Sigma$ defines a surface function. We now derive a local expansion of this function using the surface divergence operator. Through the affine map \eqref{AffineMap1}, $\phi$ induces a corresponding function on $\Sigma_s$:
\begin{equation}\label{EQ_SurfaceFunctionNew}
\tilde\phi(\bs) := \phi(\bx(\bs)) = \phi(\bx_0 + A \bs), \quad \bs \in \Sigma_s.
\end{equation}
We next compute $\nabla_s \cdot (\tilde\phi \bs)$. From the chain rule, we have
\begin{equation}\label{EQ_1022_001}
\partial_{s_j} \tilde\phi = \sum_{i=1}^n \partial_{x_i} \phi \frac{\partial x_i}{\partial s_j} = \sum_{i=1}^n \partial_{x_i} \phi \, a_{ij} = \vec{v}_j \cdot \nabla_x \phi.
\end{equation}
Therefore, it follows that
\begin{equation}\label{EQ_1022_003}
\begin{aligned}
\nabla_s \cdot (\bs \tilde\phi) &= \sum_{j=1}^m \partial_{s_j}(s_j \tilde\phi) \\
&= m\tilde\phi + \sum_{j=1}^m s_j \partial_{s_j} \tilde\phi \\
&= m\tilde\phi + \sum_{j=1}^m s_j (\vec{v}_j \cdot \nabla_x \phi) \\
&= m\tilde\phi + (A\bs)^\top \nabla_x \phi \\
&= m\tilde\phi + (\bx - \bx_0)^\top \nabla_x \phi \\
&= m\tilde\phi + \nabla_x \phi : (\bx - \bx_0).
\end{aligned}
\end{equation}

To further expand the term $\nabla_x \phi : (\bx - \bx_0)$, we replace $\phi$ in \eqref{eq:grad_step} with $\nabla_x \phi : (\bx - \bx_0)$ to obtain
\begin{equation}\label{EQ_1022_004}
\nabla_s \cdot ( \bs \nabla_x \phi : (\bx - \bx_0) ) = m \nabla_x \phi : (\bx - \bx_0) + \nabla_x(\nabla_x \phi : (\bx - \bx_0)) : (\bx - \bx_0).
\end{equation}
Next, we evaluate the term $\nabla_x(\nabla_x \phi : (\bx - \bx_0)) : (\bx - \bx_0)$. Since the differentiation is taken in $\mathbb{R}^n$ with respect to $\bx$, it follows from \eqref{eq:grad_step} (with $\bz_0 = \bx_0$ and $m=1$) that
\[
\nabla_x(\nabla_x \phi : (\bx - \bx_0)) : (\bx - \bx_0) = \nabla_x^2 \phi : (\bx - \bx_0)^2 + \nabla_x \phi : (\bx - \bx_0).
\]
Substituting this relation into \eqref{EQ_1022_004} yields
\begin{equation}\label{EQ_1022_005}
\nabla_s \cdot ( \bs \nabla_x \phi : (\bx - \bx_0) ) = (m+1) \nabla_x \phi : (\bx - \bx_0) + \nabla_x^2 \phi : (\bx - \bx_0)^2.
\end{equation}

\begin{lemma}
For any integer $k \geqslant 0$, the identity \eqref{EQ_1022_003} admits the following higher-order extension:
\begin{equation}\label{EQ_1022_006}
\nabla_s \cdot ( \bs \nabla_x^k \phi : (\bx - \bx_0)^k ) = m \nabla_x^k \phi : (\bx - \bx_0)^k + \nabla_x(\nabla_x^k \phi : (\bx - \bx_0)^k) : (\bx - \bx_0).
\end{equation}
Combining \eqref{EQ_1022_006} with the identity in Theorem \ref{thm:grad_identity} results in
\begin{equation}\label{EQ_1022_006_new}
\nabla_s \cdot ( \bs \nabla_x^k \phi : (\bx - \bx_0)^k ) = (k+m) \nabla_x^k \phi : (\bx - \bx_0)^k + \nabla_x^{k+1} \phi : (\bx - \bx_0)^{k+1}.
\end{equation}
\end{lemma}

By recursively applying \eqref{EQ_1022_006_new}, one obtains an expansion of the surface function $\tilde\phi$ on the flat surface $\Sigma$.

\begin{theorem}\label{Thm_SurfaceExpansion}
Let $\tilde\phi$ be a smooth surface function defined on an $m$-dimensional surface $\Sigma \subset \mathbb{R}^n$ through the affine mapping \eqref{EQ_SurfaceFunctionNew}. Then, for any integer $\ell \geqslant 0$, the following identity holds:
\begin{equation}\label{EQ_1022_008}
\frac{\tilde\phi(\bs)}{(m-1)!} = \sum_{k=0}^\ell \frac{(-1)^k}{(k+m)!} \nabla_s \cdot (\bs \nabla_x^k \phi : (\bx - \bx_0)^k) + \frac{(-1)^{\ell+1}}{(\ell+m)!} \nabla_x^{\ell+1} \phi : (\bx - \bx_0)^{\ell+1}.
\end{equation}
Moreover, if $\phi \in C^\infty(\mathbb{R}^n)$, then
\begin{equation}\label{EQ_1022_019}
\frac{\tilde\phi(\bs)}{(m-1)!} = \sum_{k=0}^\infty \frac{(-1)^k}{(k+m)!} \nabla_s \cdot (\bs \nabla_x^k \phi : (\bx - \bx_0)^k),
\end{equation}
provided that the series converges.
\end{theorem}

Theorem \ref{Thm_SurfaceExpansion} provides a surface-based functional expansion for the restriction of a smooth function $\phi$ onto a flat $m$-dimensional surface $\Sigma$. It expresses $\tilde\phi$ in terms of repeated applications of the surface divergence operator acting on combinations of $\phi$ and its spatial derivatives. Conceptually, this result is the surface analog of the functional expansion in Euclidean space $\mathbb{R}^n$ shown in Theorem \ref{thm:JW_expansion}. Here, the role of spatial derivatives and powers of $(\bx - \bz_0)$ is replaced by surface divergence and $(\bx - \bx_0)$, naturally respecting the geometry of $\Sigma$. Each successive term captures higher-order variations of $\phi$ along the surface directions, providing a systematic framework for constructing high-order surface approximations and integral identities.

\section{Numerical Integration in Multidimensional Domains}\label{Section:6}

Let $T$ be an $n$-dimensional polytope in $\mathbb{R}^n$. Integrating the expansion \eqref{eq:JW_series} over $T$ and applying the divergence theorem term by term yields:
\begin{equation}\label{JW_Expansion_Integral}
\begin{aligned}
\int_T \frac{f(\bx)}{(n-1)!} \, \dT &= \sum_{k=0}^\infty \frac{(-1)^k}{(n+k)!} \int_T \nabla \cdot \left( (\bx-\bz_0) \nabla^k f : (\bx-\bz_0)^k \right) \, \dT \\
&= \sum_{k=0}^\infty \frac{(-1)^k}{(n+k)!} \int_{\partial T} \left( \nabla^k f : (\bx-\bz_0)^k \right) (\bx-\bz_0) \cdot \bn \, \dS,
\end{aligned}
\end{equation}
where $\bn$ denotes the unit outward normal vector on the boundary $\partial T$.

Truncating the series \eqref{JW_Expansion_Integral} after the first $m+1$ terms yields the following approximation for the integral:
\begin{equation}\label{JW_Expansion_Integral_app}
\int_T \frac{f(\bx)}{(n-1)!} \, \dT \approx I_{m+1} := \sum_{k=0}^m \frac{(-1)^k}{(n+k)!} \int_{\partial T} \left( \nabla^k f : (\bx-\bz_0)^k \right) (\bx-\bz_0) \cdot \bn \, \dS.
\end{equation}
From the identity \eqref{eq:JW_finite}, the corresponding remainder (or truncation error) is given by:
\begin{equation}\label{JW_Expansion_Integral_Remainder}
R_{m+1} = \frac{(-1)^{m+1}}{(n+m)!} \int_T \nabla^{m+1} f : (\bx-\bz_0)^{m+1} \, \dT.
\end{equation}

The lowest-order approximation occurs when $m=0$ in \eqref{JW_Expansion_Integral_app}, yielding:
\begin{equation}\label{JW_Expansion_Integral_app_lowestorder}
\int_T f(\bx) \, \dT \approx I_{1} := \frac{1}{n} \int_{\partial T} f(\bx) \left[ (\bx-\bz_0) \cdot \bn \right] \, \dS,
\end{equation}
with the associated remainder:
\begin{equation}\label{JW_Expansion_Integral_Remainder_lowestorder}
E_{1} = -\frac{1}{n} \int_T \nabla f \cdot (\bx-\bz_0) \, \dT.
\end{equation}

For a polytope where the boundary $\partial T$ consists of a finite set of flat faces $\{F_i\}_{i=1}^F$, each with a constant outward unit normal $\bn_i$ and a constant value of $(\bx - \bz_0) \cdot \bn_i$, the surface integral in \eqref{JW_Expansion_Integral_app} simplifies to:
\begin{equation}\label{eq:surfaceintegral}
\int_{\partial T} \left( \nabla^k f : (\bx-\bz_0)^k \right) (\bx-\bz_0) \cdot \bn \, \dS =  \sum_{i=1}^F (\bx_i - \bz_0) \cdot \bn_i \int_{F_i} \nabla^k f : (\bx-\bz_0)^k \, \dS,
\end{equation}
where $\bx_i \in F_i$ is an arbitrary point on the face.

\subsection{Volume Calculation}
The volume (or $n$-dimensional measure) of the polytope $T$, denoted by $|T|$, can be computed exactly by setting $f(\bx) = 1$ in the formula \eqref{JW_Expansion_Integral_app_lowestorder}. Since $f$ is constant, $\nabla f = \mathbf{0}$, and the remainder term $E_1$ vanishes. Thus, the volume is given by:
\begin{equation}\label{EQ_Volume_Formula}
|T| = \int_T 1 \, \dT = \frac{1}{n} \int_{\partial T} (\bx - \bz_0) \cdot \bn \, \dS.
\end{equation}
Applying \eqref{eq:surfaceintegral} with $k=0$ to the surface integral in \eqref{EQ_Volume_Formula} yields:
\begin{equation}
|T| = \frac{1}{n} \sum_{i=1}^F (\bx_i - \bz_0) \cdot \bn_i  |F_i|,
\end{equation}
where $\bx_i \in F_i$ is an arbitrary point on the face, and $|F_i|$ is the $(n-1)$-dimensional measure of face $F_i$. If $\bz_0$ is chosen such that it lies on the hyperplane containing one or more faces, the corresponding terms in the summation vanish, further simplifying the calculation.

\subsection{Center of Mass Calculation}
The center of mass (or centroid) $\bar{\bx}_T$ of the polytope $T$ is defined as:
\begin{equation}
\bar{\bx}_T = \frac{1}{|T|} \int_T \bx \, \dT.
\end{equation}
To evaluate the integral of the coordinate function $f(\bx) = \bx$, we note that $\nabla \bx = \mathbb{I}$, where $\mathbb{I}$ is the $n \times n$ identity tensor. Substituting $f(\bx) = \bx$ into the identity \eqref{JW_Expansion_Integral_app_lowestorder} and accounting for the remainder term \eqref{JW_Expansion_Integral_Remainder_lowestorder}, we have:
\begin{equation}
\int_T \bx \, \dT = \frac{1}{n} \int_{\partial T} \bx \left[ (\bx - \bz_0) \cdot \bn \right] \, \dS - \frac{1}{n} \int_T \mathbb{I} \cdot (\bx - \bz_0) \, \dT.
\end{equation}
Expanding the remainder term (the last integral), we obtain:
\begin{equation}
\int_T \bx \, \dT = \frac{1}{n} \int_{\partial T} \bx \left[ (\bx - \bz_0) \cdot \bn \right] \, \dS - \frac{1}{n} \int_T \bx \, \dT + \frac{1}{n} \int_T \bz_0 \, \dT.
\end{equation}
Rearranging the terms and noting that $\int_T \bz_0 \, \dT = \bz_0 |T|$, we arrive at the following boundary integral formula for the centroid:
\begin{equation}\label{EQ_Centroid_Boundary}
\bar{\bx}_T = \frac{1}{(n+1)|T|} \left( \int_{\partial T} \bx \left[ (\bx - \bz_0) \cdot \bn \right] \, \dS + \bz_0 |T| \right).
\end{equation}
By setting $\bz_0=\mathbf{0}$, this simplifies to:
\begin{equation}\label{EQ_Centroid_Boundary_2}
\bar{\bx}_T = \frac{1}{(n+1)|T|} \int_{\partial T} \bx (\bx\cdot \bn) \, \dS.
\end{equation}
Notice that the scalar product $\bx\cdot\bn$ is constant on each flat face $F_i$. Hence, we can pull it out of the integral for each face:
\begin{equation}\label{EQ_Centroid_Boundary_3}
\bar{\bx}_T = \frac{1}{(n+1)|T|}  \sum_{i=1}^F (\bx_i\cdot\bn_i)  \int_{F_i} \bx \, \dS,
\end{equation}
where $\bx_i$ is any point on $F_i$. Let $\bx_{F_i}$ be the barycenter of the face $F_i \subset \partial T$, defined such that $\int_{F_i} \bx \, \dS = |F_i| \bx_{F_i}$. The centroid formula can then be written as:
\begin{equation}\label{EQ_Centroid_Boundary_4}
\bar{\bx}_T = \frac{1}{(n+1)|T|}  \sum_{i=1}^F (\bx_{F_i}\cdot\bn_i) \bx_{F_i}  |F_i|.
\end{equation}
This formula determines the $n$-dimensional center of mass purely through $(n-1)$-dimensional boundary data, consistent with the reduction-of-dimension strategy developed in the previous sections.

\subsection{Integration of Polynomials}
The numerical integration formula \eqref{JW_Expansion_Integral_app} is particularly powerful when $f(\bx)$ is a polynomial. Let $f(\bx) = (\bx - \bz_0)^\alpha$ be a multivariate monomial of degree $|\alpha|$. 
Using Euler's homogeneous function theorem, we note that $\nabla f \cdot (\bx-\bz_0) = |\alpha| f(\bx)$. Substituting this into $I_1$ in \eqref{JW_Expansion_Integral_app_lowestorder} and accounting for the remainder term \eqref{JW_Expansion_Integral_Remainder_lowestorder}, we have:
\begin{equation}
\int_T (\bx-\bz_0)^\alpha \, \dT = \frac{1}{n} \int_{\partial T} (\bx-\bz_0)^\alpha \left[ (\bx - \bz_0) \cdot \bn \right] \, \dS - \frac{|\alpha|}{n} \int_T  (\bx - \bz_0)^\alpha \, \dT.
\end{equation}
Moving the volume integral term to the left-hand side yields:
\begin{equation*}
\int_T (\bx-\bz_0)^\alpha \, \dT = \frac{1}{n+|\alpha|} \int_{\partial T} (\bx-\bz_0)^\alpha \left[ (\bx - \bz_0) \cdot \bn \right] \, \dS.
\end{equation*}
Again, noting that $(\bx - \bz_0) \cdot \bn$ is constant on each face $F_i$, we have:
\begin{equation}\label{eq:poly_integral}
\int_T (\bx-\bz_0)^\alpha \, \dT = \frac{1}{n+|\alpha|} \sum_{i=1}^F  \left[ (\bx_{F_i} - \bz_0) \cdot \bn_i \right]\int_{F_i} (\bx-\bz_0)^\alpha \, \dS.
\end{equation}
This identity demonstrates that the volume integral of a polynomial of degree $|\alpha|$ reduces to a sum of boundary integrals of the same degree. By applying this result recursively, integration over an $n$-dimensional polytope $T$ is eventually reduced to evaluations at the vertices (0-dimensional boundaries).

\subsection{Recursive Trapezoidal Rule in Higher Dimensions}
Using the lowest-order approximation $I_1$ from \eqref{JW_Expansion_Integral_app_lowestorder}, we can derive a systematic numerical quadrature. Observe that the remainder term $E_1$ in \eqref{JW_Expansion_Integral_Remainder_lowestorder} vanishes for any linear function $f$ if $\bz_0$ is chosen as the centroid $\bar{\bx}_T$ of the polytope $T$. Therefore, the lowest-order approximation $I_1$ is better expressed by:
\begin{equation}\label{JW_Expansion_Integral_app_lowestorder_1}
\int_T f(\bx) \, \dT \approx J_{1} := \frac{1}{n} \int_{\partial T} f(\bx) \left[ (\bx-\bar{\bx}_T) \cdot \bn \right] \, \dS,
\end{equation}
with the associated remainder:
\begin{equation}\label{JW_Expansion_Integral_Remainder_lowestorder-1}
E_{1} = -\frac{1}{n} \int_T \nabla f \cdot (\bx-\bar{\bx}_T) \, \dT.
\end{equation}

Approximating the boundary integral over each face $F_i \subset \partial T$ using the function value at its barycenter $\bx_{F_i}$ yields:
\begin{equation}
\int_T f(\bx) \, \dT \approx \frac{1}{n} \sum_{i=1}^F |F_i| f(\bx_{F_i}) [(\bx_{F_i} - \bar{\bx}_T) \cdot \bn_i].
\end{equation}
This formulation can be interpreted as a multidimensional trapezoidal rule. 

To improve accuracy, one can apply the higher-order terms from \eqref{JW_Expansion_Integral_app}. The general framework for a $k$-th order trapezoidal rule on a polytope involves:
\begin{enumerate}
    \item Decomposing the boundary $\partial T$ into its $(n-1)$-dimensional facets.
    \item Applying the surface-based functional expansion \eqref{EQ_1022_008} to each facet.
    \item Recursively reducing the dimension until the underlying integrals are expressed as weighted sums of function values (and derivatives) at the geometric centroids of the $m$-dimensional faces ($0 \le m < n$).
\end{enumerate}
This approach naturally respects the polyhedral geometry and avoids the need for complex internal triangulation (e.g., tetrahedralization) of the domain $T$.

\subsection{Integration on Flat Surfaces}
The surface integral on $F_i$ in \eqref{eq:surfaceintegral} can be further simplified by applying the functional expansion \eqref{EQ_1022_008} to each flat surface $F_i$. 

Consider a general $m$-dimensional flat surface $\Sigma \subset \mathbb{R}^n$ (with $m \le n$) and let $\tilde{\phi}$ be a smooth surface function defined on $\Sigma$ through the affine mapping \eqref{EQ_SurfaceFunctionNew}. For example, in the context of \eqref{eq:surfaceintegral}, the function $\tilde{\phi}$ corresponds to the restriction of $\phi = \nabla^k f : (\bx-\bz_0)^k$ to the surface $F_i$.  

Fix any given reference point $\bx_0 \in \Sigma$. Using the expansion \eqref{EQ_1022_008} for $\tilde{\phi}$ about $\bx_0$, we integrate both sides over the reference domain $\Sigma_s$ to obtain:
\begin{equation}\label{EQ_1022_009}
\begin{aligned}
\int_{\Sigma_s } \frac{\tilde{\phi}(\bs)}{(m-1)!} \, d\Sigma_s &= \sum_{k=0}^\infty \frac{(-1)^k}{(k+m)!}  \int_{\Sigma_s} \nabla_s \cdot \left( \bs \nabla_x^k\phi : (\bx-\bx_0)^k \right) \, d\Sigma_s \\
&= \sum_{k=0}^\infty \frac{(-1)^k}{(k+m)!}  \int_{\partial\Sigma_s} \left( \nabla_x^k\phi : (\bx-\bx_0)^k \right) \bs \cdot \bn_s \, d(\partial\Sigma_s),
\end{aligned}
\end{equation}
where $\bn_s$ denotes the unit outward normal vector on $\partial\Sigma_s$.

Equation \eqref{EQ_1022_009} follows directly from the surface divergence theorem, which converts the integral of a surface divergence over $\Sigma_s$ into a boundary integral over $\partial\Sigma_s$. This transformation expresses the surface integral of $\tilde{\phi}$ as a series of boundary contributions involving the tensor derivatives $\nabla_x^k\phi$ and the geometric moments $(\bx-\bx_0)^k$. Such a representation often simplifies computation, particularly when the boundary $\partial\Sigma_s$ has a simple geometric structure, as in polygonal or polyhedral cases. The remainder of this subsection is devoted to simplifying the following boundary integrals:
\begin{equation*}
\int_{\partial\Sigma_s} \left( \nabla_x^k\phi : (\bx-\bx_0)^k \right) \bs \cdot \bn_s \, d(\partial\Sigma_s).
\end{equation*}

Recalling the affine map from $\Sigma_s$ to $\Sigma$:
\begin{equation}\label{MapFromStoX}
\bx-\bx_0 = A \bs \quad \Longrightarrow \quad \bs = (A^\top A)^{-1} A^\top (\bx-\bx_0),
\end{equation}
and substituting this into \eqref{EQ_1022_009} gives:
\begin{equation}\label{EQ_1022_010}
\begin{aligned}
&\frac{1}{(m-1)!} \int_{\Sigma_s} \tilde{\phi}(\bs) \, d\Sigma_s \\ 
&= \sum_{k=0}^\infty \frac{(-1)^k}{(k+m)!} \int_{\partial\Sigma_s} \left( \nabla_x^k\phi : (\bx-\bx_0)^k \right) (A^\top A)^{-1} A^\top (\bx-\bx_0) \cdot \bn_s \, d(\partial\Sigma_s).
\end{aligned}
\end{equation}

For simplicity of notation, we retain only the first one or two terms of the series:
\begin{equation}\label{EQ_1022_011}
\int_{\Sigma_s } \tilde{\phi}(\bs) \, d\Sigma_s 
= \frac{(m-1)!}{m!} \int_{\partial\Sigma_s} \phi \left[ (A^\top A)^{-1} A^\top (\bx-\bx_0) \cdot \bn_s \right] \, d(\partial\Sigma_s) + \cdots
\end{equation}

We now map the integral in \eqref{EQ_1022_011} back to $\Sigma$. Using the affine transformation \eqref{MapFromStoX} and the change of variables formula, we have:
\begin{equation}\label{EQ_1022_012}
\begin{aligned}
\int_{\Sigma} \phi \, d\Sigma 
&= \frac{1}{m} \frac{|\Sigma|}{|\Sigma_s|} \int_{\partial\Sigma} \phi \left[ (A^\top A)^{-1} A^\top (\bx-\bx_0) \cdot \bn_s \right] \frac{|\partial\Sigma_s|}{|\partial\Sigma|} \, d(\partial\Sigma) + \cdots \\
&= \frac{1}{m} \frac{|\Sigma|}{|\Sigma_s|} \int_{\partial\Sigma} \phi \left[ (\bx-\bx_0) \cdot \left( A (A^\top A)^{-1} \bn_s \right) \right] \frac{|\partial\Sigma_s|}{|\partial\Sigma|} \, d(\partial\Sigma) + \cdots
\end{aligned}
\end{equation}

\begin{theorem}
Let $\Sigma$ be the image of $\Sigma_s$ under the affine map defined by $A$. The following geometric identities hold:
\begin{equation}\label{EQ_1023_001}
\frac{|\Sigma|}{|\Sigma_s|} = \sqrt{\det(A^\top A)},
\end{equation}
and on each flat portion of $\partial\Sigma$, the vector
\begin{equation}\label{EQ_1023_002}
\bn_x = A (A^\top A)^{-1} \bn_s \frac{|\partial\Sigma_s|}{|\partial\Sigma|} \frac{|\Sigma|}{|\Sigma_s|}
\end{equation}
is a unit vector that is outward normal to $\Sigma$ on $\partial \Sigma$ (see Figure \ref{FlatFaceSigma}).
\end{theorem}

Substituting the result \eqref{EQ_1023_002} into the expansion \eqref{EQ_1022_012}, the integral reduces to:
\begin{equation}\label{EQ_1022_015}
\begin{aligned}
\int_{\Sigma} \frac{\phi}{(m-1)!} \, d\Sigma &= \frac{1}{m!} \int_{\partial\Sigma} \phi \left[ (\bx-\bx_0) \cdot \bn_x \right] \, d(\partial\Sigma) \\
&\quad - \frac{1}{(m+1)!} \int_{\partial\Sigma} \left( \nabla_x\phi : (\bx-\bx_0) \right) \left[ (\bx-\bx_0) \cdot \bn_x \right] \, d(\partial\Sigma) \\
&\quad + \frac{1}{(m+2)!} \int_{\partial\Sigma} \left( \nabla_x^2\phi : (\bx-\bx_0)^2 \right) \left[ (\bx-\bx_0) \cdot \bn_x \right] \, d(\partial\Sigma) + \cdots
\end{aligned}
\end{equation}
Note that the scalar quantity $(\bx-\bx_0) \cdot \bn_x$ remains constant on each flat portion of $\partial\Sigma$. 

The expansion \eqref{EQ_1022_015} serves as a theoretical foundation for the development of numerical quadrature formulas for surface integrals. Such formulas may be constructed either by performing numerical integration over each flat face of $\partial\Sigma$, or by recursively reducing the integration to the boundaries of $\partial\Sigma$ and so on. The latter strategy is particularly effective for problems formulated in higher-dimensional spaces. Further analytical details and implementation aspects will be reported elsewhere.

\section{\texorpdfstring{$m$}{m}-Dimensional Volume of a Parallelepiped}\label{Section:7}

Let $n \ge m \ge 1$. Given $m$ linearly independent vectors 
$$
\mathbf{v}_1, \mathbf{v}_2, \ldots, \mathbf{v}_m \in \mathbb{R}^n, 
$$
and a base point (vertex) $\mathbf{x}_0 \in \mathbb{R}^n$, the $m$-dimensional parallelepiped generated by these vectors is defined as:
\begin{equation}\label{EQ_1023_800}
P_m := \left\{ \mathbf{x}_0 + \sum_{i=1}^m s_i \mathbf{v}_i : \ 0 \le s_i \le 1, \ \forall i=1,\ldots, m \right\}.
\end{equation}
Geometrically, the point $\mathbf{x}_0$ is a vertex of $P_m$, while the vectors $\{\mathbf{v}_i\}$ represent the edges emanating from $\mathbf{x}_0$. The shape of $P_m$ is an $m$-dimensional ``box'' lying in the affine subspace:
$$
\mathbf{x}_0 + \operatorname{span}\{\mathbf{v}_1, \mathbf{v}_2, \ldots, \mathbf{v}_m \}.
$$

\begin{theorem}
Let $P_m$ be the $m$-dimensional parallelepiped defined by \eqref{EQ_1023_800}, and let 
$$
A = [\mathbf{v}_1, \mathbf{v}_2, \ldots, \mathbf{v}_m]
$$ 
be the $n \times m$ matrix whose columns are the generating vectors. Then, the $m$-dimensional volume is given by:
$$
\operatorname{Vol}_m(P_m) = \sqrt{\det(A^\top A)}.
$$ 
Consequently, the identity \eqref{EQ_1023_001} holds.
\end{theorem}

\begin{proof}
Without loss of generality, assume $\mathbf{x}_0 = \mathbf{0}$. Let $\mathbb{V}_m = \operatorname{span}\{\mathbf{v}_1, \ldots, \mathbf{v}_m\}$, and let $\mathbb{V}_m^\perp$ be its orthogonal complement in $\mathbb{R}^n$. Choose an orthonormal basis for $\mathbb{V}_m^\perp$, denoted by $\{\mathbf{v}_{m+1}, \ldots, \mathbf{v}_n\}$.

Construct an auxiliary $n$-dimensional parallelepiped $P_n$ using the full set of vectors:
\begin{equation}\label{EQ_1023_801}
P_n := \left\{ \sum_{i=1}^n s_i \mathbf{v}_i : \ 0 \le s_i \le 1, \ \forall i=1,\ldots, n \right\}.
\end{equation}
Since the additional vectors $\{\mathbf{v}_{m+1}, \ldots, \mathbf{v}_n\}$ are orthonormal and orthogonal to $\mathbb{V}_m$, they contribute unit length in orthogonal directions. Thus:
\begin{equation}\label{EQ_1023_802}
\operatorname{Vol}_n(P_n) = \operatorname{Vol}_m(P_m).
\end{equation}

Let $\tilde{A} = [A, \mathbf{v}_{m+1}, \ldots, \mathbf{v}_n]$ be the $n \times n$ matrix representing the linear map from the unit hypercube $[0,1]^n$ to $P_n$. The volume is given by the determinant:
\begin{equation}\label{EQ_1023_803}
\operatorname{Vol}_n(P_n) = |\det(\tilde{A})| = \sqrt{\det(\tilde{A}^\top \tilde{A})}.
\end{equation}
Computing the product $\tilde{A}^\top \tilde{A}$ explicitly:
$$
\tilde{A}^\top \tilde{A} = 
\begin{bmatrix}
A^\top \\
[\mathbf{v}_{m+1} \dots \mathbf{v}_n]^\top
\end{bmatrix}
\begin{bmatrix}
A & [\mathbf{v}_{m+1} \dots \mathbf{v}_n]
\end{bmatrix}
=
\begin{bmatrix}
A^\top A & \mathbf{0} \\
\mathbf{0} & I_{n-m}
\end{bmatrix}.
$$
It follows that $\det(\tilde{A}^\top \tilde{A}) = \det(A^\top A) \cdot \det(I_{n-m}) = \det(A^\top A)$. Combining this with \eqref{EQ_1023_802} and \eqref{EQ_1023_803} completes the proof.
\end{proof}

\section{Transformation of Normal Vectors}\label{Section:8}

Let $\Sigma \subset \mathbb{R}^n$ be an $m$-dimensional polytope characterized by the affine map $F: \Sigma_s \to \Sigma$ defined by:
$$
\mathbf{x} = F(\mathbf{s}) = \mathbf{x}_0 + \sum_{j=1}^m s_j \mathbf{v}_j, \quad \mathbf{s} = (s_1, \ldots, s_m) \in \Sigma_s,
$$
where $\Sigma_s \subset \mathbb{R}^m$ is a regular polytope in the parameter space $\mathbb{R}^m$.

The goal of this section is to verify the geometric identity \eqref{EQ_1023_002}. Specifically, we show that the vector
\begin{equation}\label{EQ_1023_002_new}
\mathbf{n}_x = A (A^\top A)^{-1} \mathbf{n}_s \frac{|\partial\Sigma_s|}{|\partial\Sigma|} \frac{|\Sigma|}{|\Sigma_s|}
\end{equation}
is the unit outward normal vector to $\partial\Sigma$ within the affine subspace of $\Sigma$. Using the volume formula \eqref{EQ_1023_001}, this can be rewritten as:
\begin{equation}\label{EQ_1023_002_new_02}
\mathbf{n}_x = \alpha A (A^\top A)^{-1} \mathbf{n}_s, \quad \text{where } \alpha = \frac{|\partial\Sigma_s|}{|\partial\Sigma|} \sqrt{\det(A^\top A)}.
\end{equation}

\begin{lemma}
The vector $\mathbf{n}_x$ defined in \eqref{EQ_1023_002_new_02} is the outward normal to $\Sigma$ on the boundary $\partial\Sigma$.
\end{lemma}

\begin{proof}
Since $\mathbf{n}_x$ lies in the column space of $A$, it lies in the tangent space of the flat surface $\Sigma$. We must check its orthogonality to the boundary $\partial\Sigma$.

Let $\mathbf{n}_s$ be the unit outward normal at $\mathbf{s}^* \in \partial\Sigma_s$, and let $\mathbf{x}^* = \mathbf{x}_0 + A\mathbf{s}^*$. For any point $\mathbf{x} = \mathbf{x}_0 + A\mathbf{s} \in \partial\Sigma$ on the same flat face as $\mathbf{x}^*$, the dot product is:
$$
(\mathbf{n}_x, \mathbf{x} - \mathbf{x}^*) = \alpha \left( A(A^\top A)^{-1} \mathbf{n}_s \right) \cdot \left( A(\mathbf{s} - \mathbf{s}^*) \right).
$$
Using the property $(\mathbf{u}, A\mathbf{v}) = (A^\top \mathbf{u}, \mathbf{v})$, we obtain:
$$
(\mathbf{n}_x, \mathbf{x} - \mathbf{x}^*) = \alpha \left( (A^\top A)^{-1} \mathbf{n}_s \right)^\top (A^\top A) (\mathbf{s} - \mathbf{s}^*) = \alpha \mathbf{n}_s \cdot (\mathbf{s} - \mathbf{s}^*) = 0,
$$
since $\mathbf{n}_s \perp (\mathbf{s} - \mathbf{s}^*)$. Thus, $\mathbf{n}_x$ is normal to the boundary. Furthermore, for an interior point $\mathbf{s}_{\text{in}}$, we have $\mathbf{n}_s \cdot (\mathbf{s}_{\text{in}} - \mathbf{s}^*) \le 0$, which implies $(\mathbf{n}_x, \mathbf{x}_{\text{in}} - \mathbf{x}^*) = \alpha(\mathbf{n}_s, \mathbf{s}_{\text{in}} - \mathbf{s}^*)\le 0$. Hence, $\mathbf{n}_x$ is outward-pointing.
\end{proof}

To ensure $\mathbf{n}_x$ is a unit vector (i.e., $\|\mathbf{n}_x\|^2 = 1$), the scaling factor must satisfy:
\begin{equation}\label{EQ_1023_701}
\frac{|\partial\Sigma|^2}{|\partial\Sigma_s|^2} = \det(A^\top A) \left( (A^\top A)^{-1} \mathbf{n}_s, \mathbf{n}_s \right).
\end{equation}
Using the identity $\det(M)M^{-1} = \operatorname{Cof}(M)^\top$ and the symmetry of $A^\top A$, this is equivalent to:
\begin{equation}\label{EQ_1023_709}
\frac{|\partial\Sigma|^2}{|\partial\Sigma_s|^2} = \left( \operatorname{Cof}(A^\top A) \mathbf{n}_s, \mathbf{n}_s \right).
\end{equation}

\begin{theorem}
The identity \eqref{EQ_1023_701}, and equivalently \eqref{EQ_1023_709}, holds for any flat face of the polytope.
\end{theorem}

\begin{proof}
Let $\sigma_s \subset \partial\Sigma_s$ be a flat face of the boundary in the parameter space, and let $\sigma \subset \partial\Sigma$ be its corresponding image in $\mathbb{R}^n$. Select a fixed point $\mathbf{s}_\sigma \in \sigma_s$.

We construct a specific basis for $\mathbb{R}^m$ to simplify the determinant calculations.
First, choose an orthonormal basis for the tangent space of the flat face $\sigma_s$:
$$
\{\boldsymbol{\phi}_1, \ldots, \boldsymbol{\phi}_{m-1}\}.
$$
We augment this set by introducing a specific vector $\boldsymbol{\phi}_m$ defined by:
$$
\boldsymbol{\phi}_m = \frac{\boldsymbol{\beta}}{\mathbf{n}_s^\top \boldsymbol{\beta}}, \quad \text{where } \boldsymbol{\beta} = (A^\top A)^{-1}\mathbf{n}_s.
$$
Note that since $\mathbf{n}_s$ is the unit normal to the face $\sigma_s$, it is orthogonal to all tangent vectors $\boldsymbol{\phi}_1, \ldots, \boldsymbol{\phi}_{m-1}$. We can decompose $\boldsymbol{\phi}_m$ relative to this normal direction:
\begin{equation}\label{Phi_Decomposition}
\boldsymbol{\phi}_m = \mathbf{n}_s + \boldsymbol{\beta}_\tau, \quad \text{where } \boldsymbol{\beta}_\tau \in \operatorname{span}\{\boldsymbol{\phi}_1, \ldots, \boldsymbol{\phi}_{m-1}\}.
\end{equation}
The coefficient of $\mathbf{n}_s$ is exactly $1$ because $\mathbf{n}_s^\top \boldsymbol{\phi}_m = (\mathbf{n}_s^\top \boldsymbol{\beta}) / (\mathbf{n}_s^\top \boldsymbol{\beta}) = 1$. Consequently, the squared norm is:
\begin{equation}\label{Phi_Norm}
\|\boldsymbol{\phi}_m\|^2 = \|\mathbf{n}_s\|^2 + \|\boldsymbol{\beta}_\tau\|^2 = 1 + \|\boldsymbol{\beta}_\tau\|^2.
\end{equation}

Now, let $B$ be the $m \times m$ change-of-basis matrix:
$$
B = [B_{m-1}, \boldsymbol{\phi}_m], \quad \text{where } B_{m-1} = [\boldsymbol{\phi}_1, \ldots, \boldsymbol{\phi}_{m-1}].
$$
The affine map $\mathbf{s} \mapsto \mathbf{s}_\sigma + B\tilde{\mathbf{s}}$ transforms the polytope $\Sigma_s$ into another polytope $\tilde\Sigma_s$, and correspondingly maps $\sigma_s\subset\partial\Sigma_s$ to $\tilde\sigma_s\subset\partial\tilde\Sigma_s$. For any $\mathbf{s}\in\sigma_s$, the vector $\mathbf{s}-\mathbf{s}_\sigma$ is tangent to the flat face $\sigma_s$, hence $\mathbf{s}-\mathbf{s}_\sigma=\sum_{j=1}^{m-1} t_j \boldsymbol{\phi}_j$. It follws that $|\sigma_s| = |\tilde\sigma_s|$; i.e., this map preserves the $(m-1)$-dimensional measure of the face $\sigma_s$.

We now compute the determinant of $B$ by analyzing $B^\top B$:
\begin{equation}
B^\top B = 
\begin{bmatrix}
B_{m-1}^\top \\ \boldsymbol{\phi}_m^\top
\end{bmatrix}
\begin{bmatrix}
B_{m-1} & \boldsymbol{\phi}_m
\end{bmatrix}
= 
\begin{bmatrix}
I_{m-1} & B_{m-1}^\top \boldsymbol{\phi}_m \\
\boldsymbol{\phi}_m^\top B_{m-1} & \|\boldsymbol{\phi}_m\|^2
\end{bmatrix}.
\end{equation}
To diagonalize this matrix, we introduce the lower-triangular elimination matrix $L$:
$$
L = 
\begin{bmatrix}
I_{m-1} & \mathbf{0} \\
-\boldsymbol{\phi}_m^\top B_{m-1} & 1
\end{bmatrix}.
$$
Multiplying $L$ by $B^\top B$ performs row operations to eliminate the bottom-left block:
$$
L (B^\top B) = 
\begin{bmatrix}
I_{m-1} & B_{m-1}^\top \boldsymbol{\phi}_m \\
\mathbf{0} & -\boldsymbol{\phi}_m^\top B_{m-1} B_{m-1}^\top \boldsymbol{\phi}_m + \|\boldsymbol{\phi}_m\|^2
\end{bmatrix}.
$$
We simplify the bottom-right scalar term. Using the decomposition \eqref{Phi_Decomposition}, we observe that $B_{m-1}^\top \boldsymbol{\phi}_m = B_{m-1}^\top (\mathbf{n}_s + \boldsymbol{\beta}_\tau) = B_{m-1}^\top \boldsymbol{\beta}_\tau$.
Since $\boldsymbol{\beta}_\tau$ lies in the span of the columns of $B_{m-1}$, and columns of $B_{m-1}$ are orthonormal, the projection preserves the length:
$$
\boldsymbol{\phi}_m^\top B_{m-1} B_{m-1}^\top \boldsymbol{\phi}_m = \|B_{m-1}^\top \boldsymbol{\beta}_\tau\|^2 = \|\boldsymbol{\beta}_\tau\|^2.
$$
Substituting this and \eqref{Phi_Norm} into the matrix expression:
$$
-\|\boldsymbol{\beta}_\tau\|^2 + (1 + \|\boldsymbol{\beta}_\tau\|^2) = 1.
$$
Thus, the matrix reduces to an upper triangular form with $1$ on the diagonal:
$$
L (B^\top B) = 
\begin{bmatrix}
I_{m-1} & B_{m-1}^\top \boldsymbol{\phi}_m \\
\mathbf{0} & 1
\end{bmatrix} =: U.
$$
Since $\det(L)=1$ and $\det(U)=1$, we conclude that $\det(B^\top B) = 1$, and therefore $|\det(B)| = 1$.

Finally, consider the combined transformation matrix $\tilde{A} = AB$. Using the determinant product rule:
$$
\det(\tilde{A}^\top \tilde{A}) = \det(B^\top A^\top A B) = \det(A^\top A) |\det(B)|^2 = \det(A^\top A).
$$
Alternatively, we can write $\tilde{A}$ using the column partition:
$$
\tilde{A} = [A B_{m-1}, \ A \boldsymbol{\phi}_m] = [Q, \ A \boldsymbol{\phi}_m].
$$
Here, $Q = A B_{m-1}$ is the matrix mapping the parameter face $\sigma_s$ to the physical face $\sigma$. Therefore, $\det(Q^\top Q)$ represents the squared ratio of the areas:
$$
\frac{|\sigma|^2}{|\sigma_s|^2} = \det(Q^\top Q).
$$
Calculating $\tilde{A}^\top \tilde{A}$ in terms of $Q$:
$$
\tilde{A}^\top \tilde{A} = 
\begin{bmatrix}
Q^\top Q & 0 \\
0 & (A\boldsymbol{\phi}_m)^\top (A\boldsymbol{\phi}_m)
\end{bmatrix}.
$$
From our choice of $\boldsymbol{\phi}_m$, the last column $A\boldsymbol{\phi}_m$ is orthogonal to the columns of $Q$ (the tangent vectors of $\Sigma$). This block-diagonal structure implies:
$$
\det(\tilde{A}^\top \tilde{A}) = \det(Q^\top Q) \cdot ( \boldsymbol{\phi}_m^\top A^\top A \boldsymbol{\phi}_m ).
$$
Substituting $\boldsymbol{\phi}_m = \frac{(A^\top A)^{-1} \mathbf{n}_s}{\mathbf{n}_s^\top (A^\top A)^{-1} \mathbf{n}_s}$, the scalar term simplifies to:
$$
\boldsymbol{\phi}_m^\top A^\top A \boldsymbol{\phi}_m = \frac{1}{\mathbf{n}_s^\top (A^\top A)^{-1} \mathbf{n}_s}.
$$
Equating the two expressions for $\det(\tilde{A}^\top \tilde{A})$:
$$
\det(A^\top A) = \det(Q^\top Q) \frac{1}{\mathbf{n}_s^\top (A^\top A)^{-1} \mathbf{n}_s}.
$$
Rearranging for the area ratio $\det(Q^\top Q)$ yields the desired result:
$$
\frac{|\sigma|^2}{|\sigma_s|^2} = \det(A^\top A) \left( \mathbf{n}_s^\top (A^\top A)^{-1} \mathbf{n}_s \right),
$$
which leads to \eqref{EQ_1023_701}. This completes the proof of the theorem.
\end{proof}

Let us illustrate the validity of \eqref{EQ_1023_709} by using the flat face $\Sigma$ shown in Figure \ref{FlatFaceSigma}. Consider the line segment $A_1A_4$ on the boundary of $\Sigma$, with endpoints 
$$
A_1 = \bx_0 + \mathbf{v}_1 s_{11} + \mathbf{v}_2 s_{12},\quad A_4 = \bx_0 + \mathbf{v}_1 s_{21} + \mathbf{v}_2 s_{22}.
$$
It follows that
$$
A_4-A_1 = \mathbf{v}_1 s_1^* + \mathbf{v}_2 s_2^*, \ s_1^*=s_{21}-s_{11},\ s_2^*=s_{22}-s_{12}.
$$
The unit nomal vector $\bn_s$ to the corresponding edge $A_1A_4$ in the $\bs$-space is given by
$$
\bn_s = (s_2^*, -s_1^*)^\top/\sqrt{(s_1^*)^2+(s_2^*)^2}.
$$ 
On the edge $A_1A_4$, one easily obtains
\begin{equation}
\begin{split}
\frac{|\partial\Sigma|^2}{|\partial\Sigma_s|^2}  = &\frac{\|A_4-A_1\|^2}{(s_1^*)^2+(s_2^*)^2}\\
=&\frac{\|\mathbf{v}_1\|^2 (s_1^*)^2 + 2s_1^*s_2^* \mathbf{v}_1\cdot\mathbf{v}_2 + \|\mathbf{v}_2\|^2 (s_2^*)^2}{(s_1^*)^2+(s_2^*)^2}\\
=& (\Cof(A^\top A)\bn_s, \bn_s),
\end{split}
\end{equation}
which verifies \eqref{EQ_1023_709} in this special case.

%\nocite{*}
\bibliographystyle{plain}   % or use "abbrv", "siam", "unsrt", etc.
\bibliography{reference}   % references.bib file (without .bib extension)
\end{document}